\documentclass[onefignum,onetabnum,final]{siamart220329}
\usepackage[margin=1.5in]{geometry}
\usepackage{amsmath}
\usepackage{amsfonts}
\usepackage{amssymb}
\usepackage{mathrsfs}
\usepackage{bm}
\usepackage{enumerate}
\usepackage{xcolor}

\usepackage[protrusion=true,tracking=false,kerning=true]{microtype}

\usepackage{comment}

\usepackage[shortlabels]{enumitem}

\newcommand{\deletethis}[1]{{}}
\newcommand{\defeq}{\overset{\mathrm{def}}{=}}

\newcommand{\revision}[1]{{\textcolor{black}{#1}}}

\newcommand{\vaf}{{\alpha}}
\newcommand{\vbeta}{{\beta}}
\newcommand{\vlam}{{\lambda}}
\newcommand{\vups}{{\upsilon}}
\newcommand{\vA}{A}
\newcommand{\vb}{b}
\newcommand{\vc}{c}
\newcommand{\ve}{e}
\newcommand{\vs}{s}

\newcommand{\vv}{v}
\newcommand{\vw}{w}
\newcommand{\vx}{x}
\newcommand{\vX}{X}
\newcommand{\vy}{y}
\newcommand{\vz}{z}
\newcommand{\vzero}{0}

\newcommand{\C}{\mathcal{C}}
\newcommand{\Cc}{\C^\circ}
\newcommand{\Cs}{\C^*}
\newcommand{\Csc}{(\Cs)^\circ}
\newcommand{\N}{\mathbb{N}}
\newcommand{\R}{\mathbb{R}}
\newcommand{\bS}{\mathbb{S}}
\newcommand{\Rm}{\R^m}
\newcommand{\Rn}{\R^n}
\newcommand{\Rmn}{\R^{m \times n}}
\newcommand{\K}{\mathcal{K}}
\newcommand{\Kc}{\K^\circ}
\newcommand{\Ks}{{\K^*}}
\newcommand{\Ksc}{{(\Ks)^\circ}}
\newcommand{\KAs}{{\K_A^*}}
\newcommand{\T}{\mathrm{T}}

\newcommand{\cE}{\mathcal{E}}
\newcommand{\cP}{\mathcal{P}}
\newcommand{\cR}{\mathcal{R}}
\newcommand{\SAB}{\mathcal{S}_{\mathscr{A},\mathscr{B}}}
\newcommand{\Kca}{{\K_{\vc,\vA}}}
\newcommand{\Kcas}{{\K^*_{\vc,\vA}}}

\newcommand{\bd}{\operatorname{bd}}
\newcommand{\supp}{\operatorname{supp}}
\newcommand{\BC}{\mathcal{B}}
\newcommand{\HC}{\mathcal{H}}

\newcommand{\fai}{p_{\vaf_i}}

\newcommand{\norm}[1]{\left\|#1\right\|}

\newcommand{\lrp}[1]{\left(#1\right)}

\newcommand{\st}{\textup{s.t.}}

\newtheorem{example}[theorem]{Example}

\newcommand{\TheTitle}{Dual certificates of primal cone membership}
\newcommand{\TheAuthors}{Joonyeob Lee, D{\'a}vid Papp, Anita Varga}

\headers{Dual membership certificates}{\TheAuthors}

\title{{\TheTitle}\thanks{\textbf{Funding:} This material is based upon work supported by the National Science Foundation under Grant No. DMS-1847865. This material is based upon work supported by the Air Force Office of Scientific Research under award number FA9550-23-1-0370.}}

\author{
    Joonyeob LEE\thanks{North Carolina State University, Department of Mathematics. Email: \email{jlee223@ncsu.edu}.}
\and
    D{\'a}vid PAPP\thanks{North Carolina State University, Department of Mathematics. (ORCID: 0000-0003-4498-6417) Email: \email{dpapp@ncsu.edu}.}
\and
    Anita VARGA\thanks{North Carolina State University, Department of Mathematics. (ORCID: 0000-0001-7138-2921) Email: \email{avarga@ncsu.edu}.}
}

\ifpdf
\hypersetup{
	pdftitle={\TheTitle},
	pdfauthor={\TheAuthors}
}
\fi

\begin{document}
\maketitle
\begin{abstract}
\revision{We discuss optimization problems over convex cones in which membership is difficult to verify directly. In the standard theory of duality, vectors in the dual cone $\Ks$ are associated with separating hyperplanes and interpreted as certificates of  \emph{non-membership} in the primal cone $\K$. Complementing this perspective, we develop easily verifiable certificates of \emph{membership} in $\K$ given by vectors in $\Ks$. Assuming that $\Ks$ admits an efficiently computable logarithmically homogeneous self-concordant barrier, every vector in the interior of $\K$ is associated with a full-dimensional cone of efficiently verifiable membership certificates. Consequently, rigorous certificates can be computed using numerical methods, including interior-point algorithms. The proposed framework is particularly well-suited to optimization over low-dimensional linear images of higher dimensional cones: we argue that these problems can be solved by optimizing directly over the (low-dimensional) dual cone, circumventing the customary lifting that introduces a large number of auxiliary variables. As an application, we derive a novel closed-form formula for computing exact primal feasible solutions from suitable dual feasible solutions; as the dual solutions approach optimality, the computed primal solutions do so as well. To illustrate the generality of our approach, we show that the new certification scheme is applicable to virtually every tractable subcone of nonnegative polynomials commonly used in polynomial optimization (such as SOS, SONC, SAGE and SDSOS polynomials, among others), facilitating the computation of rigorous nonnegativity certificates using numerical algorithms.}
%We discuss easily verifiable cone membership certificates, that is, certificates proving relations of the form $b\in \K$ for convex cones $\K$, that consist of vectors in the dual cone $\Ks$. In the standard theory of duality, vectors in $\Ks$ are associated with separating hyperplanes and interpreted as certificates of \emph{non-membership} in $\K$. Complementing this, we present constructive certification schemes through which members of the dual cone can be interpreted as primal \emph{membership} certificates. Every vector in the interior of $\K$ is assigned a full-dimensional cone of certificates, making the numerical computation of rigorous certificates easy, provided that the dual cone has an efficiently computable logarithmically homogeneous self-concordant barrier. Of particular interest are cones that are low-dimensional linear images of much higher dimensional cones. In the context of optimization (as opposed to feasibility) problems, these certificates can be used to easily compute, using a closed-form formula, exact primal feasible solutions from suitable dual feasible solutions. As the dual solutions approach optimality, so do the computed primal solutions. We demonstrate that the new certification scheme is applicable to virtually every tractable subcone of nonnegative polynomials commonly used in polynomial optimization (such as SOS, SONC, SAGE and SDSOS polynomials, among others), facilitating the computation of rigorous nonnegativity certificates using numerical algorithms.
\end{abstract}



\section{Introduction}

Consider a closed, convex, full-dimensional, pointed cone $\K\subset\Rm$ and a vector $b\in\Rm$. If $b\not\in\K$, then there is a hyperplane strongly separating $b$ from $\K$: there exists a vector $y\in\Rm$ (the normal vector of this hyperplane) satisfying
\begin{equation*}
	\vy^\T \vb < 0 < \vy^\T \vz
\end{equation*}
for every $z\in\K\setminus\{\vzero\}$ \cite[Chapter 6]{Guler2010}. This vector $y$ can be interpreted as a \emph{non-membership certificate}, that is, a proof that $b$ does not belong to $\K$. By definition, $y$ is an element of the interior of the dual cone of $\K$,
\[\Ks = \{\vy\in\Rm\,|\;\forall\vx\in\K:\vx^\T\vy\geq 0\} \;;\]
this gives rise to the standard duality theory of convex (conic) programming. In this paper we show that with the help of self-concordance theory, \emph{membership} (rather than non-membership) in the primal cone can be similarly certified by simply exhibiting suitable vectors from the dual cone. These certificates are not unique; on the contrary, each element of the interior of $\K$ has a full-dimensional cone of membership certificates. This facilitates their computation using numerical methods and has important theoretical consequences.

Our first main contribution is that whenever $\Ks$ is a hyperbolicity cone \revision{(see \cite{FarautKoranyi1994, Koecher1999} for the definition and standard examples)}, membership in $\K$ can be certified in a manner similar to the dual sums-of-squares certificates of \cite{DavisPapp2022} (we call these \emph{H-certificates} in \cref{def:dual-certificates-B-and-H}), but the generalization does not go any further (\cref{ex:ExpConeNoH}). Then we present a new certification scheme (termed \emph{B-certificates}) that is equally efficient for the same cones as H-certificates, but can also be applied to all cones $\K$ as long as $\Ksc$ is the domain of an efficiently computable logarithmically homogeneous self-concordant barrier function (\cref{thm:full-dimensional-certificate}).

We also present a unified algorithmic framework to efficiently compute these dual membership certificates. In Sections \ref{sec:Dikin} and \ref{sec:computing}, we argue that \emph{short-step} interior-point algorithms relying on small neighborhoods are the natural approach to find such certificates.

Of particular interest are cones $\K_A$ that are the linear images of a ``simpler'' but much higher dimensional cone $\K$. Suppose that
\begin{equation}\label{eq:Ax-cone}
\K_A \defeq \{Ax\,|\,x\in\K\},
\end{equation}
where $A\in\Rmn$ with $m \ll n$ is a matrix with linearly independent rows. On the surface, solving an optimization problem over $\K_A$, or even checking whether a given vector $b$ belongs to $\K_A$, requires the solution of an optimization problem over $\K$, and in order to certify $b\in\K_A$, we must present an $x\in\K$ satisfying $b=Ax$. This raises two issues: first, this ``primal'' membership certificate $x$ is $n$-dimensional, potentially much larger than the vector whose membership is being certified, which may be prohibitive even when $m=\dim(\K_A)$ is small. Second, if a numerical method is used to compute $x$, the computed certificate is likely only an approximate certificate: $\|b-Ax\|$ is small but nonzero. If we need an exact certificate (e.g., a rational vector $x \in \K \cap \mathbb{Q}^n$ with $b=Ax$), additional post-processing is required to ``round'' or project the inexact numerical certificate to an exact one, taking care that this additional step does not result in the certificate leaving the cone $\K$, which is a highly non-trivial task in general, especially if $\vA$ is ill-conditioned or the numerically computed $\vx$ is close to the boundary of $\K$. For example, in the context of semidefinite programming representations of sums-of-squares cones, this problem was studied in detail in \cite{PeyrlParrilo2008, KaltofenLiYangZhi2012}.

In Section \ref{sec:projected-cones}, we demonstrate that our dual membership certificates overcome both of these issues. First, because the cone of certificates for every vector is full-dimensional, numerically computed certificates are automatically exact, rigorous certificates; and since they are just elements of $\K_A^*$, they are not higher dimensional than the certified vector. Second, the certification schemes come with a \emph{reconstruction formula} for cones of the form \eqref{eq:Ax-cone} that allows the efficient computation of a primal membership certificate $x\in\K$ that by definition satisfies $b=Ax$ exactly. (Theorem \ref{thm:reconstruction}.) For most cones $\K$ that commonly arise in conic optimization problems, the formula automatically yields rational primal certificates $x$ (as long as $A$ is a rational matrix), which can also be verified efficiently in rational arithmetic.

In Section \ref{sec:optimization}, we show how dual certificates can be used to compute, with a similar closed-form formula, exact primal feasible solutions of conic optimization problems in standard form from suitable dual feasible solutions, with a guarantee that the closer the dual solutions are to optimality, the closer the computed primal feasible solutions will be to optimality. With this formula, most short-step, path-following interior-point algorithms, including dual-only methods, can be ``retrofitted'' to compute primal-dual feasible pairs with a guaranteed low duality gap with minimal additional computational cost in each iteration.

In Section \ref{sec:POP}, we outline the application of these ideas in polynomial optimization. The most commonly used inner approximations of the cones of nonnegative polynomials (including sums-of-squares (SOS) polynomials, sums of nonnegative circuit (SONC) polynomials, sums of AM-GM exponentials (aka. SAGE functions), SDSOS polynomials and others) are all cones of the form \eqref{eq:Ax-cone} with an easily computable matrix $\vA\in\mathbb{Q}^{m\times n}$, usually with $m\ll n$, and a (possibly non-symmetric) cone $\K$ for which $\Ks$ admits a well-known logarithmically homogeneous self-concordant barrier function. Therefore, the problem of rigorously certifying the nonnegativity of a polynomial using either of these cones amounts to the exact solution of a feasibility problem in $\K_\vA$, as discussed above. When $\K_\vA$ is any of the aforementioned cones of polynomials, the dual certificates introduced in this paper can be interpreted as easily verifiable rigorous polynomial nonnegativity certificates that can be computed numerically without constructing an explicit ``primal'' representation (such as an explicit representation of an SOS polynomial as a sum of squared polynomials). Should such a representation be necessary, the reconstruction formula can be used to compute an exact representation in closed form using a numerically computed dual certificate.

\smallskip

\paragraph{\bf Relation to prior work} The results in this paper are broad generalizations of prior work of Davis and Papp \cite{DavisPapp2022} on \emph{weighted sums-of-squares} (WSOS) cones, which are convex cones instrumental in global polynomial optimization. Membership in these cones can be cast as a (feasibility) semidefinite program (SDP) \cite{Parrilo2000}; every feasible solution to this SDP is an obvious membership certificate (a Gram matrix of the SOS polynomial). However, the number of variables in this SDP is at least an order of magnitude larger than the dimension of the cone, and in the presence of a large number of weights, the ratio of the dimensions of these two cones can be arbitrarily large. Instead, the certificates introduced in \cite{DavisPapp2022} (simply termed there as \emph{dual certificates}) are elements of the dual WSOS cone, and are thus low-dimensional, and can be more efficiently computed using non-symmetric cone programming methods than solving the conventional SDP. The theory presented in the present paper is a generalization of these earlier results, and the algorithms discussed here are simultaneously more general and more efficient (owing to their lower iteration complexity) than the algorithms presented in the earlier work.

In particular, the new certification scheme is applicable to virtually every tractable subcone of nonnegative polynomials commonly used in polynomial optimization besides WSOS polynomials; we outline some of these applications in Section \ref{sec:POP}.

\smallskip

\paragraph{\bf Background: convex cones and self-concordant barriers} Central to the paper are the notions of \emph{self-concordant barrier functions} and \emph{logarithmic homogeneity}. The reader may find a comprehensive treatment of this subject in the monographs \cite{NesterovNemirovskii1994,Nesterov2018,Renegar2001}; for easier referencing of key results, we recall a few fundamental definitions and properties below.

\begin{definition}
Let $\C\subseteq\Rn$ be a convex cone with a non-empty interior $\Cc$. A function $f:\Cc\to\R$ is called a \emph{strongly nondegenerate, self-concordant barrier} (SCB) for $\C$ if $f$ is strictly convex, at least three times differentiable, and has the following three properties (with $\bd(\C)$ denoting the boundary of $\C$):
\begin{enumerate}[\hspace{1em} 1.]
\item $f(x)\to\infty$  as $x\to\bd(\C)$.
\item $|D^3 f(x)[h,h,h]| \leq 2\left(D^2f(x)[h,h]\right)^{3/2}$.
\item \revision{Denoting the gradient and Hessian of $f$ by $g$ and $H$, respectively,} we have $\nu \defeq \sup_{x\in\Cc} g(x)^\T H(x)^{-1}g(x) < \infty$.
\end{enumerate}
We refer to the parameter $\nu$ above as the \emph{barrier parameter} of $f$.

We say that $f$ is a \emph{logarithmically homogeneous self concordant barrier} (LHSCB) with \emph{barrier parameter $\nu$}, or a $\nu$-LHSCB for short if, in addition to the properties above, it is also true that for every $x\in\Cc$ and $t>0$, $f(t x) = f(x)-\nu\ln(t)$.
\end{definition}

\revision{In the rest of the paper, we will continue to use $g$ and $H$ to denote the gradient and the Hessian of the SCB $f$.}

Given a positive definite, real, symmetric matrix $M$ of order $n$, we denote by $\langle \cdot,\cdot\rangle_M$ the inner product $(x,y)\mapsto x^\T My$; the induced norm is $\|\cdot\|_M$. If $M=H(x)$ is the Hessian of an LHSCB at $x\in\Cc$, we may use the shorthands $\|\cdot\|_x$ in place of $\|\cdot\|_{H(x)}$ and $\|\cdot\|^*_x$ in place of $\|\cdot\|_{H(x)^{-1}}$. These are commonly referred to as the primal and dual \emph{local norms} associated with the barrier function.

Proposition \ref{thm:SCB-properties} below summarizes a few properties of self-concordant barrier functions used throughout the paper. The proofs of these statements can be found, for example, in \cite[Chapters 2--3]{Renegar2001} and \cite[Chapter 5]{Nesterov2018}.

\begin{proposition}
\label{thm:SCB-properties} Suppose $f:\Cc\to\R$ is an SCB for some convex cone $\C\subseteq\Rn$. Then the following hold:
\begin{enumerate}[1.]
\item The gradient map $g:\Cc\to\Rn$ is a bijection between $\Cc$ and $-\Csc$. That is, for every $b\in\Csc$ there exists a unique $y_b\in\Cc$ that satisfies $-g(y_b)=b$.
\item For every $x,y\in\Cc$,

\begin{equation}
\label{eq:gTy-geq-ynormx}
 -g(x)^\T y \geq \|y\|_x.
\end{equation}
\end{enumerate}
If, in addition, $f$ is a $\nu$-LHSCB, then the following also hold:
\begin{enumerate}[label=\arabic*.]
\setcounter{enumi}{2}

\item The gradient and Hessian satisfy
\begin{equation*}
g(tx) = t^{-1}g(x) \text { and } H(tx) = t^{-2}H(x) \text{ for every } t>0,
\end{equation*}
and
\begin{equation}\label{eq:gH-identities}
    H(x)x = -g(x)\quad\text{and}\quad \|g(x)\|_{x}^* = \|x\|_{x} = \sqrt{-g(x)^\T x} = \sqrt{\nu}.
\end{equation}
\item \label{item:injective-LHSCB} If $A:\Rm\to\Rn$ is an injective linear operator and $\mathcal{A} \defeq \left\{x\in\Rm\,|\,A(x)\in\Cc\right\}$ is non-empty, then $x\mapsto f(A(x))$ is a $\nu$-LHSCB with domain $\mathcal{A}$.
\end{enumerate}
\end{proposition}

\section{Membership certificates from the dual cone}

Throughout this paper, we shall use the following assumptions and notational shortcuts: \revision{we assume that $\K$ is a \emph{proper convex cone}, i.e., it is convex, closed, full-dimensional (has a non-empty interior), and pointed; in addition, we assume that} the dual cone $\Ks$ admits a $\nu$-LHSCB $f$, and we denote the gradient and Hessian of $f$ at $y\in\Ksc$ by $g(y)$ and $H(y)$, respectively. Finally, we define
\[B(y) \defeq H(y)+g(y)g(y)^\T.\]

The membership certification schemes introduced in this section are based on the following facts. Note that the first half of the next theorem holds under the less restrictive assumption that $f$ is an SCB, rather than an LHSCB.
\begin{theorem}\label{thm:B-cert}
Let $\K$ be a proper convex cone whose dual cone $\Ks$ admits an SCB $f$. Then, with the above notation, for every $y\in\Ksc$ and $z\in\Ks$,
\begin{equation}\label{eq:BxinK}
B(y)z\in\K.
\end{equation}
If $\Ks$ is the hyperbolicity cone of the hyperbolic polynomial $p$ and $f$ is the logarithmic barrier $-\ln(p(\cdot))$, then we also have
\begin{equation}\label{eq:HxinK}
H(y)z\in\K.
\end{equation}
\end{theorem}
\begin{proof}
Using properties of the SCBs and the Cauchy--Schwarz inequality, we have that every $y\in\Ksc$, $z\in\Ks$, and $w\in\Ks$ satisfy the following chain of inequalities:
\begin{align*}
\left(g(y)^\T w\right)\left(g(y)^\T z\right) \stackrel{\eqref{eq:gTy-geq-ynormx}}{\geq} \|w\|_y \|z\|_y \geq |\langle w,z\rangle_y| = |w^\T H(y) z| \geq -w^\T H(y) z.
\end{align*}
Rearranging the left- and right-hand sides, we get
\begin{align*}
0 \leq w^\T H(y) z + \left(g(y)^\T w\right)\left(g(y)^\T z\right) = w^\T \left(H(y)+g(y)g(y)^\T\right) z = w^\T B(y) z.
\end{align*}
Since $\vw\in\Ks$ can be chosen arbitrarily, this yields \cref{eq:BxinK}, proving our first assertion.

The second claim is G{\"u}ler's result from \cite[Thm.~6.2]{Guler1997}.
\end{proof}

The assumption of hyperbolicity in the second half of \cref{thm:B-cert} is necessary; the relation \eqref{eq:HxinK} does not hold for every cone, as the following examples show.
\begin{example}[Relation \eqref{eq:HxinK} does not hold for every cone]\label{ex:ExpConeNoH}
Consider the exponential cone $\cE$ defined by

\begin{equation*}
    \cE \defeq \operatorname{cl}\left(\left\{(x_1,x_2,x_3)\in\R_+^2\times\R\,|\,x_2>0, x_1 \geq x_2 e^{x_3/x_2}\right\}\right),
\end{equation*}
where $\operatorname{cl}(\cdot)$ denotes (topological) closure. A well-known $3$-LHSCB for $\cE$ is
\begin{equation*}
f_\cE(x_1,x_2,x_3) = -\ln x_1 -\ln x_2 -\ln(x_2 \ln(x_1/x_2)-x_3);
\end{equation*}
see, e.g., \cite[Eq.~(4.5)]{Nesterov2006}. Now, it can be verified by direct computation that the following vectors all belong to $\cE^\circ$:
\begin{equation*}
y = (6,2,-3), \quad z = (2,4,-3), \quad w = (416,1,6).
\end{equation*}
With some rather tedious arithmetic we can also compute that
\begin{equation*}
w^\T H(y) z = \frac{\ln (3) (3211+904 \ln (3))-4637}{9 (3+\ln (9))^2} < -0.075,
\end{equation*}
whose negativity proves that, although $y\in\cE$ and $z\in\cE$, $H(y)z\not\in\cE^*$. Therefore, $\K = \cE^*$ does not satisfy the relation \eqref{eq:HxinK}.

For another counterexample, consider the power cone $\cP_{(2/3,1/3)}$ defined by
\[ \cP_{(2/3,1/3)} \defeq \left\{(x_1,x_2,x_3)\in\R_+^2\times\R \,|\, (x_1^2x_2)^{1/3}\geq  |x_3| \right\}. \]
A well-known 3-LHSCB for this cone is
\[ f_{P_{(2/3,1/3)}}(x_1,x_2,x_3) = -\ln\left((x_1^2x_2)^{2/3}-x_3^2\right) - \frac{1}{3}\ln(x_1) - \frac{2}{3}\ln(x_2); \]
see, e.g., \cite{RoyXiao2022}. Now, it can be verified by direct computation that the following vectors all belong to $\cP_{(2/3,1/3)}^\circ$:
\begin{equation*}
y = (10,1,1), \quad z = (1,20,2), \quad w = (355,1,50),
\end{equation*}
however,
\begin{equation*}
w^\T H(y) z = \frac{12\,871-443\,620 \sqrt[3]{10} + 195\,500 \sqrt[3]{100}}{60 \left(1-10 \sqrt[3]{10}\right)^2} < -1.399,
\end{equation*}
proving that \eqref{eq:HxinK} is not satisfied.
\end{example}

Also note that the matrices $B(y)$ and $H(y)$ are automatically invertible for every $y\in\Ksc$, since $H(y)$ is positive definite. We are now ready for our key definitions.

\begin{definition}[dual membership certificates]
\label{def:dual-certificates-B-and-H}
Let $b\in\operatorname{span}(\K)$ be given. A vector $y\in\Ksc$ is a \emph{B-certificate of $b$} if, using the notation of \cref{thm:B-cert}, $B(y)^{-1}b\in\Ks$. It is an \emph{H-certificate of $b$} if $H(y)^{-1}b\in\Ks$. If the cone and the implied type of certificate are clear from the context, we may refer to these as \emph{dual certificates}.
\end{definition}

\revision{\cref{thm:B-cert} justifies the terminology: if a vector $b\in\operatorname{span}(\K)$ has a B-certificate, then it is guaranteed to be in $\K$, and the same holds for H-certificates if $\Ks$ is a hyperbolicity cone.}

In both cases, the certificate $y$ is easy to verify as long as the first two derivatives of the barrier function are easily computable (which in turn implies that membership in $\Ks$ is easily verifiable): one needs to first verify $y\in\Ksc$ and compute the matrix $B$ or $H$ at $y$; the rest of the verification process amounts to simple linear algebra and verification of the membership of another vector in $\Ks$. In particular, many known LHSCBs have efficiently computable gradients and Hessians that are also rational at every rational point. In this case, the computed rational dual membership certificates can even be rigorously verified in rational arithmetic.

\begin{example}
\label{ex:SOS}
Sums-of-squares (and weighted sums-of-squares) cones are subcones of polynomials that are non-negative over basic, closed semi-algebraic sets; they are well-known for their instrumental role in polynomial optimization \cite{Lasserre2001}. Their dual cones are spectrahedral and thus hyperbolic, and this means that H-certificates are valid membership certificates for sums-of-squares polynomials. In fact, H-certificates generalize the concept of ``dual certificates'' introduced in the recent work \cite{DavisPapp2022} for cones of (weighted) sums-of-squares polynomials.

When polynomials are represented in many of the commonly used bases (e.g., the monomial basis or the Chebyshev basis of the first or second kind), the duals of the corresponding sums-of-squares cones admit an LHSCB of the form $y\to -\ln\det\Lambda(y)$ with an injective linear map $\Lambda$ that maps rational vectors to rational symmetric matrices. As a result, $g(y)$ is a rational vector and $H(y)$ is a rational matrix for every rational vector $y$. This yields a simple, independent proof of a theorem of Powers \cite{Powers2011} on the existence of rational SOS certificates \cite{DavisPapp2022}, and also means that rational H-certificates of SOS polynomials with rational coefficients can be efficiently verified in rational arithmetic \cite{DavisPapp2024}. We shall expound on this example in Section \ref{sec:SOS}.
\end{example}

We now show that the converse also holds: every vector in $\Kc$ has a B-certificate, and if the employed SCB is an LHSCB, then $\Ks$ contains a \emph{full-dimensional} subcone of vectors that certify a given $b\in\Kc$. The fact that these cones of certificates are full-dimensional is critical, as they allow us to compute them easily with numerical methods such those discussed in Section \ref{sec:computing}.

\begin{theorem}\label{thm:full-dimensional-certificate}
Let $b\in\Kc$, and consider its set of B- and H-certificates with respect to an LHSCB of $\Ksc$:
\begin{equation}\label{eq:BC-def}
 \BC(b) \defeq \left\{ y\in\Ksc\,\middle|\,B(y)^{-1}b\in\Ks \right\}
\end{equation}
and
\begin{equation}\label{eq:HC-def}
 \HC(b) \defeq \left\{ y\in\Ksc\,\middle|\,H(y)^{-1}b\in\Ks \right\}.
\end{equation}
Additionally, let $y_b$ be the unique vector in $\Ksc$ that satisfies $-g(y_b) = b$. (Recall \cref{thm:SCB-properties}.) Then $\BC(b)$ and $\HC(b)$ are full-dimensional subcones of $\Ksc$ containing $y_b$.
\end{theorem}
\begin{proof}
From the definitions in \eqref{eq:BC-def} and \eqref{eq:HC-def} and \cref{thm:SCB-properties} it immediately follows that both $\BC(b)$ and $\HC(b)$ are subsets of $\Ksc$ and that they are cones; all we have left to show is that they both contain every vector $y$ in some neighborhood of $y_b$.

Recall from Eq.~\eqref{eq:gH-identities} that $H(y)^{-1}g(y) = -y$ for every $y\in\Ksc$. Thus, we have
\begin{equation*}
H(y_b)^{-1}b = -H(y_b)^{-1}g(y_b) = y_b \in \Ksc,
\end{equation*}
proving that $y_b\in\HC(b)$. And because the function $y\mapsto H(y)^{-1}b$ is continuous in some neighborhood of $y_b$, it is immediate that $y\in\HC(b)$ for every $y$ in some neighborhood of $y_b$, proving the second claim.

The proof of the first claim is similar, but we need to use, in addition, the Sherman--Morrison formula and the second identity of Eq.~\eqref{eq:gH-identities} to compute
\begin{subequations}\label{eq:Binv}
\begin{align}
B(y)^{-1} &= \left(H(y)+g(y)g(y)^\T\right)^{-1} = H(y)^{-1} - \frac{H(y)^{-1}g(y)g(y)^\T H(y)^{-1}}{1+g(y)^\T H(y)^{-1}g(y)}\\
&\stackrel{\eqref{eq:gH-identities}}{=} H(y)^{-1} - \frac{1}{1+\nu}yy^\T,
\end{align}
\end{subequations}
after which we can see that
\begin{align*}
B(y_b)^{-1}b &= -B(y_b)^{-1}g(y_b) = -H(y_b)^{-1}g(y_b) + \frac{y_b\left(y_b^\T g(y_b)\right)}{1+\nu} \\ &\stackrel{\eqref{eq:gH-identities}}{=} \left(1-\frac{\nu}{\nu+1} \right)y_b \in \Ksc.
\end{align*}
That is, $y_b\in\BC(b)$, moreover, every $y$ in a neighborhood of $y_b$ is also contained in $\BC(b)$.
\end{proof}

\deletethis{
The following definition is a useful shorthand for the rest of the paper, motivated by the previous theorem.

\begin{definition} \label{def:gradient_certificate}
We shall refer to the unique vector $y_b\in\Ksc$ that satisfies $-g(y_b)=b$ as the \emph{gradient certificate} of $b$.
\end{definition}
The gradient certificate, by definition, depends not only on the vector $b$ and the cone it belongs to, but also on the barrier function whose gradient is used in the definition. However, as this barrier will always be clear from the context, we do not explicitly annotate this dependence of $y_b$ on the barrier function.
}

We close this section by showing that a B-certificate is always an H-certificate, but not vice versa. That is, the cone of H-certificates for a given vector is generally larger than the cone of B-certificates (but, unlike B-certificates, they only certify membership when $\Ks$ is the hyperbolicity cone).

\begin{lemma}
If a vector $y\in\Ksc$ is a B-certificate for $b$, then $y$ is also an H-certificate for $b$. The converse, in general, does not hold.
\end{lemma}
\begin{proof}
Assume $B(\vy)^{-1}\vb\in\Ks$. From our formula \eqref{eq:Binv} for $B(y)^{-1}$, we see that
\begin{equation*}
H(y)^{-1}b = B(y)^{-1}b + \frac{y^\T b}{1+\nu}y.
\end{equation*}
By assumption, the first term, $B(y)^{-1}b$, belongs to $\Ks$. So does the second term, $\frac{y^\T b}{1+\nu}y$, since $y\in\Ks$ by assumption and $b\in\K$, therefore $y^\T b\geq 0$. We conclude that $H(y)^{-1}b\in\Ks$, that is, $y$ is an H-certificate of $b$, as claimed.

Example \ref{ex:ExpConeNoH} shows that the converse in general does not hold. For some cones that are not hyperbolicity cones, H-certificates exist even for some vectors $b\not\in\K$ (that is, they are not membership certificates in general cones), while we have seen that B-certificates only exist for vectors in $\K$.
\end{proof}

\section{A reconstruction formula for linear images of cones}
\label{sec:projected-cones}

We now turn our attention to our primary motivating examples, which are surjective linear images of high-dimensional cones. Let
\begin{equation}\label{eq:Ax-cone-with-dual}
\K_A = \{Ax\,|\,x\in\K\} \text{ and therefore }\; \KAs = \{y\,|\,A^\T y\in\Ks\},
\end{equation}
where \revision{$A\in\Rmn$}, the rows of $A$ are linearly independent, and assume that $\Ks$ has a non-empty interior that is the domain of a known, easily computable LHSCB. Even if $\K$ has a known LHSCB, its linear image $\K_A$ does not ``inherit'' an obvious barrier function; this is the main reason why it is difficult to optimize over $\K_A$ and why deciding feasibility in $\K_A$ is non-trivial even if deciding feasibility in $\K$ is easy. However, under mild conditions, $\KAs$ inherits an easily computable LHSCB from $\Ks$. \revision{More precisely, using Proposition \ref{thm:SCB-properties}, if $f$ is an LHSCB for $\Ks$, $A^\T$ is injective, and $\KAs$ is full-dimensional, then the function $f_A:(\KAs)^\circ\to\R$ defined by
\begin{equation}\label{eq:def-f}
f_A \defeq y \mapsto f(A^\T y)
\end{equation}
is automatically an LHSCB for $\KAs$.} The gradient and Hessian of $f_A$ can be expressed in closed form.

\begin{lemma}
Let $g$ and $H$ denote the gradient and Hessian of $f$. Then the gradient $g_A$ and Hessian $H_A$ of the function $f_A$ defined in \eqref{eq:def-f} can be calculated as follows
\begin{equation}\label{eq:gH_H}
g_A(y) = A g(A^\T y) \qquad H_A(y) = A H(A^\T y)A^\T.
\end{equation}
In the same vein,
\begin{equation}\label{eq:gH_B}
B_A(y) = A B(A^\T y)A^\T.
\end{equation}
\end{lemma}

Our next theorem gives a closed-form formula with which we can compute a suitable pre-image of $b\in\K_A$ that automatically belongs to $\K$. This makes the verification of $b\in\K_A$ even simpler (as long as membership in $\K$ can be easily verified).

\begin{theorem}[Reconstruction Theorem]
\label{thm:reconstruction}
Suppose $b\in\K_A^\circ$, and let $y\in \BC(b)$ be one of its dual membership certificates. Then the vector
\begin{equation}\label{eq:reconstruction-formula}
x = H(A^\T y)A^\T\!\left(AH(A^\T y)A^\T\right)^{-1}b
\end{equation}
satisfies $Ax=b$ and $x\in\K$.
%The same statement holds for H-certificates if in \eqref{eq:reconstruction-formula} we replace every instance of the matrix $B$ with $H$.
\end{theorem}
\begin{proof}
We first note that $x$ is well-defined: $H(A^{\T}y)$ is well-defined and invertible, since $y\in(\K_A^*)^\circ$ and $A^\T y\in(\Ks)^\circ$; additionally, $A$ is of full row-rank, which ensures that the matrix $\left(AH(A^\T y)A^\T\right)$ inverted in the formula is indeed invertible.

The first of our two claims, $Ax=b$, is immediate from simple matrix arithmetic.

To show that $x\in\K$, first note $B_A(y)$ is also invertible. Then we can apply Eqs.~\eqref{eq:gH_H}--\eqref{eq:gH_B} and \cref{thm:B-cert} to the cone $\KAs$ and its LHSCB $f_A$ and see that
\begin{equation*}
B_A(y)^{-1}b = (AB(A^\T y)A^\T)^{-1}b \in \KAs.
\end{equation*}
Then, by definition, we have
\begin{equation*}
A^\T(AB(A^\T y)A^\T)^{-1}b \in \Ks,
\end{equation*}
and using \cref{thm:B-cert} again, we obtain the relation
\begin{equation*}
B(A^\T y) A^\T(AB(A^\T y)A^\T)^{-1}b \in \K.
\end{equation*}
Finally, note that
\[B(A^\T y) A^\T(AB(A^\T y)A^\T)^{-1} = H(A^\T y) A^\T(AH(A^\T y)A^\T)^{-1}\]
by repeated application of \eqref{eq:gH-identities} in Proposition \ref{thm:SCB-properties}.
\end{proof}

Note that Theorem \ref{thm:reconstruction} does not require a separate formula for H-certificates, since (for cones that admit H-certificates) every H-certificate is automatically a B-certificate.

The formula \eqref{eq:reconstruction-formula} converts the dual certificate $y$ into a conventional, ``primal'' certificate of $b$'s membership in $\K_A$. We note, however, that the relation $b\in\K_A$ is already rigorously certified by the relation $B_A(y)^{-1}b \in \KAs$, and the computation of the primal certificate $x$ is not necessary to produce or verify a dual certificate. On the contrary, dual certificates allow us to optimize directly over $\K_A$ and verify the feasibility of the computed (near-)optimal solutions without resorting to computing and verifying the membership of any vector in $\K$, which is useful whenever $\dim(\K_A)\ll\dim(\K)$. We detail the algorithmic aspects of these questions in Section \ref{sec:computing}.

\smallskip

\noindent \textbf{Primal reconstruction and weighted least-squares.} The formula \eqref{eq:reconstruction-formula} can be interpreted as the optimal solution of the equality-constrained (weighted) least-squares problem
\begin{equation}
\label{eq:least-squares-x}
    \min_x \left\{\left\|-g(A^\T y)-x\right\|_{H(A^\T y)^{-1}}\,\middle|\, Ax=b\right\},
\end{equation}
where $y\in\Ksc$ is a given vector. Straightforward arithmetic involving \eqref{eq:gH_H} (without any reliance on Proposition~\ref{thm:SCB-properties}) shows that if $\lambda$ denotes the Lagrange multipliers associated with the equality constraints, then the optimality conditions of \eqref{eq:least-squares-x} yield the optimal multipliers
\[\lambda^* = y-H_A(y)^{-1}b\]
and the optimal (primal) solution
\begin{equation}
\label{eq:least-squares-x-OC-x}
x^* = H(A^\T y)A^\T H_A(y)^{-1}b,
\end{equation}
exactly as in the reconstruction formula \eqref{eq:reconstruction-formula}.

The interpretation of \eqref{eq:least-squares-x} is fairly clear. The range of $-g$ is $\Kc$, in which we wish $x$ to belong (recall the first claim of Proposition \ref{thm:SCB-properties}); thus, this problem seeks to find an $x$ satisfying $Ax=b$ that is close (in a suitable, $y$-dependent norm) to a given vector that is already in $\Kc$.

For further insight, note that the optimal value of \eqref{eq:least-squares-x} is
\[ \left\|-g(A^\T y)-x^*\right\|_{H(A^\T y)^{-1}} = \left\|-g_A(y)-b\right\|_y^*, \]
keeping in mind that the dual local norm $\|\cdot\|^*_y$ now corresponds to the barrier $f_A$.

We can use these observations to paint a more geometric picture, and to derive easy-to-use sufficient conditions for dual certification. For this, we briefly recall the notion of Dikin ellipsoids.

\begin{definition}
Given a proper convex cone $\C\subseteq\R^n$ whose interior is the domain of an LHSCB, we define the local (open) ball centered at $x\in\Cc$ with radius $r$ as $\mathscr{B}(x,r) \defeq \{v\in\R^n\,|\, \|v-x\|_{x} < r\}$. These balls are called \emph{Dikin ellipsoids}. In the same vein, for a $y\in\Csc$ we have the dual Dikin ellipsoid of radius $r$ defined as $\mathscr{B}^*(y,r) \defeq \{v\in\R^n\,|\, \|v-y\|_{y}^* < r\}$.
\end{definition}

It is well-known that Dikin ellipsoids of radius $1$ stay entirely within their respective cones \cite[Chapter 2]{Renegar2001}. Applying this to $\K_A$ and $\K$, we see that if
\begin{equation}
\label{eq:dualcert-dikin}
    \min_x \left\{\left\|-g(A^\T y)-x\right\|_{H(A^\T y)^{-1}}\,\middle|\, Ax=b\right\} = \left\|-g_A(y)-b\right\|_y^* < 1,
\end{equation}
then $b\in\K_A$, moreover, if $\Ks$ is a hyperbolicity cone, this is also proven by the fact that its preimage $x$ given by the reconstruction formula belongs to $\K$. That is, the condition \eqref{eq:dualcert-dikin} is a sufficient (but not necessary!) condition for a dual vector to be an H-certificate of $b$.

We can state this result more generally, for arbitrary convex cones instead of cones of the form $\K_A$. In the next section, we explore this connection between Dikin ellipsoids and H- and B-certificates further.

\section{Dikin ellipsoids and dual certificates}
\label{sec:Dikin}

Throughout this section, we consider again a general proper convex cone $\K$, assuming only that $\Ksc$ is the domain of a $\nu$-LHSCB. %As before, the gradient and Hessian of this barrier function are denoted by $g$ and $H$, respectively.

\begin{theorem}
\label{thm:Dikin-sufficient-H}
Suppose $y\in\Ksc$ satisfies the relation $\|-g(y) - b\|_y^* < 1$ or (equivalently, using the primal local norm) $\|y - H(y)^{-1}b\|_y < 1$. Then $y$ is an H-certificate for $b\in\K$.
\end{theorem}
\begin{proof}
Identical to the last argument of the previous section.
\end{proof}

The converse is not true in general, that is, the cone generated by Dikin ellipsoids may be a proper subcone of the cone of H-certificates. We can also use Dikin ellipsoids to derive analogous sufficient conditions for a vector to be a B-certificate:

\deletethis{
\begin{definition}[Dikin certificates]
\label{def:dualcert-dikin}
We say that $x\in\Ksc$ is an \emph{ellipsoidal Dikin certificate} of $b$ if $b \in \mathscr{B}^*(-g(x),1)$. This relation can also be written, equivalently, as $\|-g(x) - b\|_x^* < 1$ or (using the primal local norm) as $\|x - H(x)^{-1}b\|_x < 1$.

Similarly, $x\in\Ksc$ is a \emph{conic Dikin certificate} of $b$ if for some $\delta>0$ it satisfies the relation $b \in \mathscr{B}^*(-g(\delta x),1)$.
\end{definition}

From the form $\|x - H(x)^{-1}b\|_x < 1$, we can immediately see that every Dikin certificate is an H-certificate:
\begin{lemma}
If $x$ is an (ellipsoidal or conic) Dikin certificate of $b\in\Ks$, then $x$ is also an H-certificate for $b\in\Ks$.
\end{lemma}
The converse is not true, that is, conic Dikin certificates are proper subcones of the cone H-certificates. We can also use Dikin ellipsoids to derive analogous sufficient conditions for a vector to be a B-certificate:
}

\begin{theorem}
\label{thm:Dikin-sufficient-B}
Let $\Ksc$ be the domain of a $\nu$-LHSCB, and suppose that $b\in\K$ and $y\in\Ksc$ satisfy the relation $b \in \mathscr{B}^*(-g(\delta y),r)$ for some $\delta>0$ and $r=\frac{1}{2(\nu+1)}$. Then $y$ is a B-certificate of $b$.
\end{theorem}
\begin{proof}
It suffices to prove our theorem for $\delta=1$, since $\BC(b)$ is a cone. We will show our claim by arguing that in this case,
\begin{equation*}
\lambda = \frac{1+\nu}{1+\nu - y^\T b}
\end{equation*}
is a positive scalar for which $\norm{y - \lambda B(y)^{-1} b}_y < 1,$ therefore $y$ is a B-certificate for $\lambda b$.

To lighten the notation, we shall use the shorthands $g$, $H$, and $B$ for $g(y)$, $H(y)$, and $B(y)$, respectively.

Note that
\[0 \leq \norm{y - H^{-1}b}_y^2 = y^\T H y + b^\T H^{-1}b - 2y^\T b.\]

    By the definition of the dual local norm and the fact that $y^\T H y \stackrel{\eqref{eq:gH-identities}}{=} \nu$, we get
    \begin{equation}
    \label{pt1}
        y^\T b
        \leq \frac{1}{2}(y^\T H y + b^\T H^{-1}b)
        = \frac{1}{2}(\nu + {\norm{b}_y^*}^2).
    \end{equation}
    Also by the assumption $b \in \mathscr{B}^*(-g(y),r)$ and the fact that $\norm{g}_y^* = \sqrt{g^\T H^{-1} g} = \sqrt{\nu}$ we obtain
    \begin{align*}
        \norm{b}_y^*
        &\leq \norm{-g}_y^* + \norm{b - (-g)}_y^*\\
        &= \norm{-g}_y^* + \norm{y - H^{-1}b}_y\\
        &\leq \sqrt{\nu} + r.
    \end{align*}
    Therefore, (\ref{pt1}) implies
    \[
        y^\T b
        \leq \frac{1}{2}(\nu + {\norm{b}_y^*}^2)
        \leq \frac{1}{2}(\nu + (\sqrt{\nu} +r)^2)
        = \nu + \sqrt{\nu} r + \frac{1}{2} r^2,
    \]
    and we get
    \begin{equation*}
        1 + \nu - y^\T b \geq 1 - \sqrt{\nu} r - \frac{1}{2} r^2 = 1- \frac{\sqrt{\nu}}{2(1+\nu)} - \frac{1}{8(1+\nu)^2} \geq \frac{23}{32} > 0.
    \end{equation*}
    Thus we can conclude that $\lambda$ is positive; indeed
    \begin{equation*}
        0 < \lambda \leq \frac{32}{23}(1+\nu).
    \end{equation*}
    From this, using \eqref{eq:Binv} and the definition of $\lambda$, we also obtain
    \begin{align*}
        \norm{y - B^{-1}\lambda b}_y
        &= \norm{y - H^{-1}\lambda b + \frac{\lambda}{1 + \nu}yy^\T b}_y\\
        &= \norm{\lrp{1 + \frac{\lambda}{1 + \nu}y^\T b}y - H^{-1}\lambda b}_y\\
        &= \lambda\norm{y - H^{-1}b}_y\\
        &\leq \frac{32}{23} \frac{1+\nu}{2(1+\nu)}  <1,
    \end{align*}
completing the proof.
\end{proof}
We underscore that \cref{thm:Dikin-sufficient-B} is only a sufficient, but not necessary, condition for a vector to be a B-certificate.

\section{Computing dual membership certificates}
\label{sec:computing}
We now turn to the question of efficiently deciding whether $b \in \Kc$ holds and computing an H- or B-certificate if it does. In addition to assuming that $\Ksc$ is the domain of an efficiently computable LHSCB and that we have a known $\vy\in\Ksc$, we also assume that a vector $w \in \Kc$ is given. (The latter is essentially without loss of generality; e.g., we can always choose $\vw=-g(y)\in\Kc$.) Given such a $w$, the vector $b - \alpha w$ also belongs to $\Kc$ for every sufficiently negative $\alpha$. Furthermore, $b \in \K$ holds if and only if there exists a nonnegative $\alpha$ such that $b - \alpha w \in \K$. Based on this observation, we can consider the following optimization problem and its dual:
\begin{equation} \label{eq:membership_opt}
\begin{aligned}
    & \max_{\vz, \alpha} & & \alpha \\
    &  \ \st & & \alpha \vw + \vz = \vb \\
    & & & \vz \in \K, \ \alpha \in \R
\end{aligned}
\qquad \qquad
\begin{aligned}
    & \min_y & & b^\T y \\
    &  \ \st & &w^\T y = 1 \\
    & & & y \in \mathcal{K}^*
\end{aligned}
\end{equation}

It is clear from our assumptions that strong duality holds with attainment on both the primal and dual side, since both problems have a Slater point. \revision{The optimal objective function value $\alpha^*$ of \eqref{eq:membership_opt} is positive if and only if $b \in \Kc$.} However, standard duality theory and an approximately optimal numerical solution of \emph{either problem} cannot rigorously certify this fact: a strictly feasible dual solution $\vy$ can be scaled to exactly satisfy the only linear constraint $\vw^\T\vy=1$, and thus we can obtain a true, verifiable dual feasible solution, however, the corresponding objective value is an \emph{upper bound} on the exact optimal value, therefore, a nonnegative dual value implies nothing. On the primal side we get a lower bound, and so an exact primal feasible $(\vz,\alpha)$ with $\alpha\geq 0$ would be a rigorous certificate, but numerical methods for problems of this form (interior-point algorithms in particular) tend to yield only approximately feasible primal solutions, meaning that only a vector close to $\vb$ is certified to be in the cone. In contrast, any $\vy\in\BC(\vb)$ (or $\vy\in\HC(\vb)$ if $\Ks$ is a hyperbolicity cone) does rigorously certify $b\in\K$. Our goal is to compute such a dual certificate $\vy$.

The problem \eqref{eq:membership_opt} is a \emph{nonsymmetric convex conic optimization problem}, with the particular difficulty that no assumptions are made about the cone $\K$. There are only a handful of primal-dual interior-point algorithms (IPAs) developed for this general class of problems, which require only one LHSCB but simultaneously compute near-optimal primal and dual solutions; see, for example, \cite{SkajaaYe2015, PappYildiz2022, PappVarga2025, CoeyKapelevichVielma2023}).\footnote{In these works it is assumed that the \emph{primal} cone $\K$ admits a tractable LHSCB, with no assumptions on the \emph{dual} cone $\Ks$. We made the opposite assumption for \eqref{eq:membership_opt} to maintain consistency with the rest of this paper. Since these methods solve the primal and dual problems simultaneously, this is only a notational difference.} All of these are suitable tools for numerically addressing the decision problem of whether $b \in \K$. However, if we also aim to obtain dual certificates, it appears that short-step IPAs that operate within a close neighborhood of the central path are the most appropriate methods. This is due to the relationship between the definition of this neighborhood and the sufficient conditions we derived for dual certificates in \cref{thm:Dikin-sufficient-H,thm:Dikin-sufficient-B}.

\deletethis{
To make this more precise, let us denote the strict primal-dual feasible set of \eqref{eq:membership_opt} by $\mathcal{F}^{\circ}$, that is,
\[ \mathcal{F}^{\circ} = \{ (z,\alpha,y) \in \Kc \times\R\times \Ksc \mid \alpha w + z = b, w^\T y = 1 \}.\]
In our recent paper \cite{PappVarga2025}, we proposed IPAs operating in the neighborhood
\begin{equation*}
    \mathcal{N}(\eta,\tau) = \left\{ (z,\alpha,y) \in \mathcal{F}^\circ : \| y - H(y)^{-1} (z/ \tau) \|_{y} \leq \eta \right\},
\end{equation*}
where $\eta \in (0,1)$ is the neighborhood radius and $\tau$ is the so-called central path parameter.
Similar neighborhoods of the central path were used, for example, in \cite{Serrano2015,SkajaaYe2015}.
}

To make this more precise, consider for a moment a general conic optimization problem in standard form
\begin{equation} \label{eq:std_conic_PD}
\begin{aligned}
    & \inf_{\vx} & & \vc^\T\vx \\
    &  \ \st & & \vA\vx = \vb \\
    & & & \vx \in \K
\end{aligned}
\qquad \qquad
\begin{aligned}
    & \sup_\vy & & \vb^\T\vy \\
    &  \ \st & &\vA^\T\vy+\vs=\vc\\
    & & & \vs\in\Ks
\end{aligned}
\end{equation}
whose strict feasible set is
\[
\mathcal{F}^{\circ} \defeq \left\{ (\vx, \vy, \vs)\in\mathcal{K}^\circ\times\mathbb{R}^m\times(\mathcal{K}^*)^\circ \,\middle|\, A\vx = \vb, \ \vA^{\T}\vy + \vs = \vc \right\}.
\]
In our recent paper \cite{PappVarga2025}, we proposed IPAs for solving \eqref{eq:std_conic_PD} that requires only an LHSCB for the \emph{primal} cone $\K$ and operates in the neighborhood
\begin{equation}
\label{eq:nbd}
\mathcal{N}(\eta,\tau) = \left\{ (\vx, \vy, \vs) \in \mathcal{F}^\circ\mid \| \vs + \tau g(\vx) \|_{\vx}^* \leq \eta \tau \right\},
\end{equation}
where $\eta \in (0,1)$ is the neighborhood radius (a fixed parameter throughout the algorithm, chosen by the user) and $\tau$ is the central path parameter (which tends geometrically to $0$ as we approach optimality). Similar neighborhoods are used, for example, in \cite{Nesterov2018,Serrano2015,SkajaaYe2015}.

When specialized to \eqref{eq:membership_opt} (keeping in mind that $\Ks$ plays the role of $\K$!), the strict feasible region and neighborhood defined above simplify to
\begin{align*}
\mathcal{F}^{\circ} &= \{ (z,\alpha,y) \in \Kc \times\R\times \Ksc \mid \alpha w + z = b, w^\T y = 1 \} \text{ and }\\
\mathcal{N}(\eta,\tau) &= \left\{ (z,\alpha,y) \in \mathcal{F}^\circ\mid \| y - H(y)^{-1} (z/ \tau) \|_{y} \leq \eta \right\}.
\end{align*}

Recall that in Section \ref{sec:Dikin}, we derived the following sufficient conditions for B- and H-certificates, which are nearly identical to the formula defining $\mathcal{N}$:
\begin{align*}
&\norm{y-H(y)^{-1}(b/\tau)}_y < 1 \qquad\quad\;\implies y\in\HC(b) \textrm{ and}\\
&\norm{y-H(y)^{-1}(b/\tau)}_y < \frac{1}{2(1+\nu)} \implies  y\in\BC(b)
\end{align*}
for any $\tau > 0$. Therefore, if we have a strictly feasible solution $(z,\alpha,y)\in\mathcal{F}^{\circ}$ to the optimization problem \eqref{eq:membership_opt} with $(z,\alpha,y) \in \mathcal{N}(\eta,\tau)$ for any $\tau > 0$ and $\eta < 1$, then the dual feasible solution $y$ is an H-certificate for $b-\alpha w$. If $\eta < 1/(2(1+\nu))$ holds, then $y$ is a B-certificate for $b-\alpha w$. Together with $\alpha \geq 0$, this certifies $b \in \K$ in both cases.

In conclusion, if $\vb\in\Kc$ and we use the IPAs proposed in \cite{PappVarga2025} (or any other polynomial IPA operating in a similar Dikin-ellipsoid based neighborhood of the central path) to solve \eqref{eq:membership_opt}, we can efficiently compute a dual membership certificate.

Another desirable property of this approach is that even dual feasible solutions to \eqref{eq:membership_opt} that are far from optimal (but are near the central path) can be meaningful, interpreting $\alpha$ as the distance from the boundary of the cone in the direction $\vw$, since every iterate $(\vz,\alpha,\vy)$ of a feasible IPA that operates within the neighborhood defined above is a rigorously certified bound $\alpha$ on the optimal value, with a dual certificate $\vy$. For instance, in the polynomial optimization problems considered in Section \ref{sec:POP}, we can choose the constant $1$ polynomial as $\vw$ and compute a sequence of rigorously certified lower bounds on a given polynomial. (We elaborate on this application in Section \ref{sec:POP}.)

\section{Near-optimal primal feasible solutions}
\label{sec:optimization}

We now turn our attention from certifying membership (and the associated optimization problem \eqref{eq:membership_opt}) to general convex optimization problems in standard conic form \eqref{eq:std_conic_PD}.

The main result of this section is that as long as we employ short-step dual or primal-dual IPAs to solve these problems (an even broader class of algorithms than those considered in the previous section), \emph{an exact feasible solution $\vx$ to the primal problem can be computed in closed form from any dual feasible solution $\vy$ near the central path}, with a guarantee that \emph{the closer $\vy$ is to optimality, the closer the computed $\vx$ is to optimality}.

To see how this can be algorithmically useful, besides being of independent theoretical interest, let us reformulate \eqref{eq:std_conic_PD} using the notation
\[ \Kca \defeq \left\{\begin{pmatrix}\vc^\T\vx\\-\vA\vx\end{pmatrix}\,\middle|\,\vx\in\K\right\} \]
and an auxiliary variable $\gamma$, as follows:
\begin{equation} \label{eq:PD_w_Kca}
\begin{aligned}
    & \inf_{\gamma\in\R} & & \gamma \\
    &  \ \st & & \begin{pmatrix}\gamma\\-\vb\end{pmatrix} \in \Kca
\end{aligned}
\qquad \qquad
\begin{aligned}
    & \sup_{\vy\in\Rm} & & \vb^\T\vy \\
    &  \ \st & & \begin{pmatrix}1\\ \vy\end{pmatrix} \in \Kcas
\end{aligned}
\end{equation}

The problems \eqref{eq:std_conic_PD} and \eqref{eq:PD_w_Kca} are almost equivalent: they have the same optimal values, and their dual feasible regions are also clearly the same; however, \eqref{eq:PD_w_Kca} eliminates the variables $\vx$ and $\vs$, at the cost that the new primal optimization problem is now over a cone $\Kca$, for which we may not have an efficiently computable LHSCB even if $\K$ has one. The dual problem also involves a new cone, $\Kcas$. If $\vA$ has many more columns than rows, then the formulation \eqref{eq:PD_w_Kca} involves cones of much lower dimension, and may be solved more efficiently, than the original problem \eqref{eq:std_conic_PD}.

In addition to our previous assumption that $\Ksc$ is the domain of a known $\nu$-LHSCB $f_\Ks$, let us make following common regularity assumptions:
\begin{enumerate}[1)]
\item Both the primal and dual problems are feasible (therefore we have finite optimal values);
\item $\begin{pmatrix}\vc^\T\\-\vA\end{pmatrix}$ has linearly independent rows; and
\item $\Kcas$ has a nonempty interior.
\end{enumerate}
The last two assumptions ensure that $(\Kcas)^\circ$ is also the domain of a known $\nu$-LHSCB (\cref{thm:SCB-properties}, Item \ref{item:injective-LHSCB}), namely
\begin{equation}
\label{eq:fKcas}
f_{\Kcas}(y_0,\vy) \defeq f_\Ks(\vc y_0 - \vA^\T\vy).
\end{equation}

\revision{Now, suppose that we have an algorithm that we can apply to \eqref{eq:PD_w_Kca} and compute, using only the dual barrier $f_{\Kcas}$, the (approximately) optimal primal and dual values and a (near-)optimal dual feasible solution $\vy$ for this problem. As the optimal objective values and the dual feasible regions of \eqref{eq:PD_w_Kca} and \eqref{eq:std_conic_PD} are the same, this also solves the dual problem of \eqref{eq:std_conic_PD} without computing any of its primal feasible solutions}. When $\vA$ has many more columns than rows, circumventing the iterative optimization of $\vx$ in this manner can save computational time and memory \cite{PappYildiz2019,PappVarga2025}. On the other hand, many applications call for the actual primal solution to \eqref{eq:std_conic_PD}, not just the optimal value. As we shall see, in that case, we can still apply the same methods, and a primal feasible and (near-)optimal solution to the original problem \eqref{eq:std_conic_PD} can be computed in post-processing using our reconstruction formula, at little additional cost. In fact, throughout the algorithms, primal feasible solutions that serve as easily verifiable certificates of our progress towards optimality can be computed in each iteration at a negligible additional cost.

Mirroring our discussion in the Introduction on feasbility certificates, one could argue that these reconstructed feasible solutions are preferable to the numerically computed ones, since they are theoretically guaranteed to be exact if computed (as post-processing) in higher-precision or exact arithmetic. Alternatively, one may use this throughout the algorithm as a guarantee that the numerical method did not veer too far from optimality due to numerical problems (or more reliably detect numerical problems), without resorting to a method that runs entirely in exact or high-precision arithmetic.

In the remainder of this section, we cover primal-dual and dual-only methods separately, as they rely on slightly different neighborhood definitions.

\subsection{Exact primal reconstruction in primal-dual IPAs}
In this section, we discuss how exact primal feasible solutions can be reconstructed from dual feasible solutions in the context of the authors' primal-dual path-following IPA in \cite{PappVarga2025} and other methods that operate in the neighborhood \eqref{eq:nbd}.

The arguments for this case are essentially identical to the ones used in Section \ref{sec:computing}, after noting that \eqref{eq:PD_w_Kca} can also be equivalently written in the form of \eqref{eq:membership_opt}, with $\gamma$ playing the role of $\alpha$, the vector $\ve_0=(1,0,\dots,0)$ playing the role of $\vw$, and $(0,-\vb)$ playing the role of $\vA$. (The only difference is that $\ve_0$ does not necessarily belong to $(\Kca)^\circ$, but this is irrelevant now.)

With this change in notation, the summary of the previous section is that the application of IPAs using the neighborhood \eqref{eq:nbd} to the optimization problem \eqref{eq:PD_w_Kca} yields a sequence of iterates $(\gamma,\vy)$, each of which has the property that $(1,\vy)$ certifies the relation $(\gamma,-\vb)\in\Kca$, that is, it certifies the existence of an $\vx\in\K$ satisfying $\vA\vx=\vb$ and $\vc^\T\vx=\gamma$ without ever computing such an $\vx$ in the process. We may compute such a primal vector $\vx$ using our reconstruction theorem (\cref{thm:reconstruction}), with $f_{\Kcas}$ in place of $f$, $(1,\vy)$ in place of $\vy$ and $(\gamma,-\vb)$ in place of $b$. And since these IPAs require the computation of the gradient and the factorization of the Hessian of the barrier in each iteration, the additional computation represented by the reconstruction formula \eqref{eq:reconstruction-formula} does not add a substantial computational cost. The bulk of the additional cost is due to matrix-vector multiplications that are cheaper than a single iteration of the IPA.

\subsection{Exact primal reconstruction in dual IPAs}
It may appear unsurprising that a primal feasible solution $\vx$ of \eqref{eq:std_conic_PD} can be reconstructed from a suitable corresponding primal feasible solution of \eqref{eq:PD_w_Kca}, but we emphasize that this is only made possible with a suitable dual certificate $\vy$. To further underscore this point, and to make our results more generally applicable, in this section we show that \emph{dual-only} methods applied to \eqref{eq:PD_w_Kca} that make no reference to the primal cone $\Kca$ can also be ``retrofitted'' to compute exact primal feasible solutions from suitable dual feasible solutions.  Our approach is applicable to all short-step, path-following, dual IPAs (or primal IPAs applied to our dual problem).

When applied to \eqref{eq:PD_w_Kca}, such methods generate a sequence of dual vectors $\vy$ in a neighborhood of the \emph{dual central path} $\{\vy_\tau\,|\,\tau>0\}$, where $\vy_\tau$ is the unique minimizer of the problem
\begin{equation}
\label{eq:ytau}
y_\tau \defeq \arg\min \left\{\tau f_{\Kcas}(y_0,\vy)-\vb^\T\vy\,\middle|\,y_0=1,(y_0,\vy)\in(\Kcas)^\circ\right\},
\end{equation}
and the neighborhood is defined in terms the local norm associated with the dual barrier. Often, the central path parameter $\tau$ in \eqref{eq:ytau} is equal to or bounded from below by the scaled duality gap $\mu =(\vc^\T\vx-\vb^\T\vy)/\nu$ at each iteration, and, crucially,  $\vb^\T\vy_\tau$ is within $\nu\tau$ of the optimal objective function value -- this is what guarantees linear convergence of the methods in the objective function. (See \cite[Theorem 5.3.10]{Nesterov2018} for a classic example.) We have the following general scheme for primal reconstruction in dual methods that follow the above principle.
\begin{theorem}
\label{thm:reconstruction-optimization}
Assume that $(1,\vy)\in(\Kcas)^\circ$ (equivalently, that $(\vy,\vs)$ is a dual feasible solution to \eqref{eq:std_conic_PD} for some $\vs\in\Kc$), and \revision{let $\gamma\in\R$ be an arbitrary scalar.} Denote the Hessian of the $\nu$-LHSCB $f_\Ks$ by $H_{\Ks}$. %, and let $B_{\Ks}=H_{\Ks}+g_{\Ks}g_{\Ks}^\T$.
Then the vector
\begin{equation}\label{eq:reconstruction-formula-optimization}
\vx_\gamma \defeq H_{\Ks}(\vups)\vA_c^\T\left(\vA_cH_{\Ks}(\vups)\vA_c^\T\right)^{-1}\vb_\gamma,
\end{equation}
where
\begin{equation*}
\vA_c = \begin{pmatrix}\vc^\T\\-\vA\end{pmatrix}
\text{, }\quad
\vb_\gamma=\begin{pmatrix}\gamma\\-\vb\end{pmatrix}
\text{, \; and }\;
\vups=\vA_c^\T\begin{pmatrix}1\\\vy\end{pmatrix}
\end{equation*}
satisfies $\vA\vx_\gamma=\vb$ and $\vc^\T\vx_\gamma=\gamma$. If, in addition, $\vy$ is sufficiently close to the unique minimizer $\vy_\tau$ of \eqref{eq:ytau} and we choose $\gamma=\vb^\T\vy+\nu\tau$, then $\vx_\gamma\in\K$, that is, $\vx_\gamma$ is a primal feasible solution of \eqref{eq:std_conic_PD} whose objective value is no more than $\nu\tau$ away from optimal.

%As in the rest of this paper, if $\Ks$ is hyperbolic and $f_\Ks$ is the corresponding logarithmic barrier, then every instance of $B_{\Ks}$ in \eqref{eq:reconstruction-formula-optimization} can be replaced by $H_{\Ks}$.
\end{theorem}
\begin{proof}

The optimality conditions of \eqref{eq:ytau} yield the following partial characterization of $\vy_\tau$: the last $m$ components of the the gradient $g(\cdot)$ of $f_{\Kcas}$ are
\[ \frac{\partial f}{\partial\vy}f_{\Kcas}(1,\vy)\Big|_{\vy=\vy_\tau} = b/\tau. \]
To determine the missing zeroth component of the gradient, $g_0 = \frac{\partial f}{\partial y_0}f_{\Kcas}$, we once again rely on Eq.~\eqref{eq:gH-identities} from \cref{thm:SCB-properties}. We have
\[ \nu = -\begin{pmatrix}g_0\\ b/\tau\end{pmatrix}^\T \begin{pmatrix} 1\\ \vy_\tau\end{pmatrix} = -g_0 - \tau^{-1} b^\T \vy_\tau, \]
therefore, the barrier gradient satisfies
\begin{equation}
\label{eq:taugtau}
-\tau g(1,\vy_\tau) = \begin{pmatrix}\vb^\T\vy_\tau +\tau\nu \\ -\vb\end{pmatrix}.
\end{equation}
Recall from the proof of \cref{thm:full-dimensional-certificate} that Eq. \eqref{eq:taugtau} implies that $(1,\vy_\tau)$ is both an H- and B-certificate of the membership of the right-hand side vector in $\Kcas$, certifying the existence of a feasible solution $\vx$ in \eqref{eq:std_conic_PD} whose attained value is $\vc^\T\vx=\vb^\T\vy_\tau +\tau\nu$. From the same theorem, every vector $\vy$ sufficiently close to $\vy_\tau$ has the same property; that is, $(1,\vy)$ certifies the existence of a primal feasible solution $\vx$ whose objective function value is $\vb^\T\vy+\tau\nu$.\footnote{It is possible to make this statement and the remaining ``for sufficiently close...'' statements of this section more precise and quantitative using \cref{thm:Dikin-sufficient-H} and \cref{thm:Dikin-sufficient-B}, but we omit the details, as some path-following methods may use different neighborhoods anyway.}

We may construct such a primal vector explicitly using our reconstruction theorem (\cref{thm:reconstruction}), with $\vA_\vc$ in place of $\vA$, $\vups$ in place of $\vy$, and the vector $\vb_\gamma$ in place of $b$. The resulting formula is precisely Eq.~\eqref{eq:reconstruction-formula-optimization}.
\end{proof}

We note that because the reconstructed $\vx_\gamma$ only depends on $\gamma$ through the term $\vb_\gamma$, the vector $\vx_\gamma$ is an affine function of $\gamma$ that can be computed as easily as if $\gamma$ was a known scalar. (See \cref{ex:LP-reconstruction} for an illustration.) Thus, we can apply \cref{thm:reconstruction-optimization} ``parametrically'' to certify, using the current dual iterate $(1,\vy)$ as the dual certificate, a generic primal pair $(\gamma,-\vb)$, where $\gamma$ is an unknown scalar parameter.

We can then proceed in one of two ways: for some cones, such as the nonnegative orthant and the positive semidefinite cone, it is possible to find the best certifiable $\gamma$ directly from this affine expression. (For linear programming this is simple arithmetic, for semidefinite programming it leads to a generalized eigenvalue problem.)

If this is not possible or efficient enough, we may invoke the sufficient conditions found in Section \ref{sec:Dikin}. Comparing Eq.~\eqref{eq:taugtau} to those results, we see that every vector $(y_0,\vy)$ in a suitable Dikin ellipsoid centered at $(1,\vy_\tau)$ is a dual certificate of the relation $(\vb^\T\vy_\tau +\tau\nu, -\vb)\in\Kcas$, certifying that there exists a feasible solution $\vx$ in \eqref{eq:std_conic_PD} whose attained value is $\vc^\T\vx=\vb^\T\vy_\tau +\tau\nu$. Conversely, $(1,\vy)$ is a dual certificate of $-g(1,\vy)$ and also of every vector in the Dikin ellipsoid of radius $r$ around this center. (Here, $r=1$ for H-certificates and $r=(2(\nu+1))^{-1}$ for B-certificates, based on Theorems \ref{thm:Dikin-sufficient-H} and \ref{thm:Dikin-sufficient-B}.) Thus, for every $\vy$ sufficiently close to $\vy_\tau$, we have that
\begin{equation}
\label{eq:gamma-sufficient}
(\gamma,-b) \in \mathscr{B}^*(-g(1,\vy),r)
\end{equation}
for some $\gamma$. With $\vy$ and the radius $r$ fixed, the set of values $\gamma$ for which
\eqref{eq:gamma-sufficient} holds
is an interval that can be easily computed by solving a scalar quadratic equation. (To set up this equation, we need a factorization or the inverse of the Hessian of $f_\Kcas$ at $y$, but this is already necessary to run the IPAs. The cost of solving this equation is therefore negligible.) What our arguments above show is that \emph{in a (known) neighborhood of the central path, the smallest $\gamma$ value in this interval} is no greater than $\vb^\T\vy_\tau +\tau\nu$, and therefore the current dual solution and the reconstructed primal solution will have a small duality gap.

\begin{example}
\label{ex:LP-reconstruction}
Consider a standard form linear program and its dual, Eq.~\eqref{eq:std_conic_PD}
with
\[  \K=\R_+^4, \quad \vc^\T = (1,1,1,1), \quad \vA = \begin{pmatrix}
    1 & 2 & 3 & 4\\ 0 & -6 & 0 & 1
\end{pmatrix}, \quad \vb = \begin{pmatrix}
    19 \\ -5
\end{pmatrix}.\]
The optimal objective function value is $5.5$, attained at $\vx=(0, 1.5, 0, 4)^\T$. A crudely calculated near-optimal dual solution near $\vy_{0.1}$, the point corresponding to $\tau=0.1$ on the central path, is
\[ \vy = (6/23, -2/29)^\T \; \text{ with }\; \vb^\T\vy = 3536/667 \approx 5.30135.\]

We can plug this dual solution (with $y_0=1$) and our barrier \eqref{eq:fKcas} into the reconstruction formula with the target primal objective function value $\gamma$. The resulting primal solution candidate is
\[\renewcommand*{\arraystretch}{1.5}
\vx_\gamma =
\begin{pmatrix}
    \frac{10459868 \gamma -57626369}{13715262} \\
    \frac{21391355-2022605 \gamma}{6857631} \\
    \frac{31571864 \gamma -173276291}{13715262} \\
    \frac{-4045210 \gamma +31353325}{2285877}
\end{pmatrix},
\]
and it can be directly verified that $\vc^\T\vx_\gamma=\gamma$ and $\vA\vx_\gamma=\vb$ independent of $\gamma$, as desired. In this simple example, rather than using the sufficient condition \eqref{eq:gamma-sufficient}, we can directly compute the smallest $\gamma$ for which $\vx_\gamma\geq 0$. The best attainable primal objective value thus certified is
\[ \gamma = 57626369/10459868\approx 5.50928. \]
Note that, perhaps surprisingly, the certified primal objective is considerably closer to optimal than the value attained by the dual feasible solution used to certify it.
\end{example}

\section{Application to polynomial optimization}
\label{sec:POP}

In this section, we explore the application of H- and B-certificates in the context of certificates proving the nonnegativity of polynomials. As it will be apparent, our tools developed in this paper can be directly applied to all the tractable nonnegativity certificates popular in the polynomial optimization literature.

Recall that a polynomial optimization problem (POP) is a problem of the form
\begin{equation}
\label{eq:POP}
\inf \left\{ p(x)\;|\;q_i(x)\geq 0\; (i=1,\dots,m) \right\},
\end{equation}
where $p$ and $q_1,\dots,q_m$ are $n$-variate polynomials. This, generally intractable, non-convex optimization problem is often approached as a (convex) optimization problem over the cone of nonnegative polynomials:
\[ \sup \left\{\alpha\in\R\;|\; p-\alpha\in\mathscr{P}_\Delta\right\} ,\]
where $\mathscr{P}_\Delta$ is the cone of polynomials that are nonnegative over the set
\begin{equation}
\label{eq:Delta}
\Delta \defeq \left\{ \vx\in\R^n\;|\;q_i(x)\geq 0\; (i=1,\dots,m) \right\}.
\end{equation}
The strategy then is to replace $\mathscr{P}_\Delta$ with a convex subcone $\mathscr{C}\subseteq\mathscr{P}_\Delta$ whose membership problem is tractable, such as a (weighted) sums-of-squares (SOS) or sums of nonnegative circuit (SONC) polynomial cone:
\begin{equation}
\label{eq:POP-relax}
\sup \left\{\alpha\in\R\;|\; p-\alpha\in\mathscr{C}\right\}.
\end{equation}
Every feasible solution of \eqref{eq:POP-relax} yields a lower bound on the optimal value of \eqref{eq:POP}. This can be computed as a heuristic (with no expectation that the resulting lower bound is close to the infimum of the original problem), or as part of a scheme (theoretically supported by a \emph{Positivstellensatz}, i.e., a characterization of positive polynomials, such as \cite{Handelman1988, Putinar1993, DresslerIlimanDeWolff2017, KuryatnikovaVeraZuluaga2024}), where a sequence of increasingly good approximations in the form of nested cones
\[ \mathscr{C}_1 \subseteq \mathscr{C}_2 \subseteq \dots
\quad \text{ satisfying }\; \operatorname{cl}\left({\cup_{i=1}^\infty \mathscr{C}_i}\right) = \mathscr{P}_\Delta\]
is available, leading to a sequence of lower bounds converging to the infimum of \eqref{eq:POP}.

In the presence of strong duality, \revision{if we were only concerned about the approximate numerical value of this lower bound}, we may opt to solve the dual problem of \eqref{eq:POP-relax} instead; for example, if $\C$ is a suitable SOS cone, \eqref{eq:POP-relax} is known as the \emph{SOS relaxation} of the POP \eqref{eq:POP}, and the dual of \eqref{eq:POP-relax} is known as  the \emph{moment relaxation} \cite{Lasserre2001}. The theoretical conundrum is now similar to the issue we have encountered in the Introduction and Section \ref{sec:optimization}: every cone $\C$ that is typically used in this context to approximate $\mathscr{P}_\Delta$, including each of those examined later in this section, is represented as a linear image $\K_A$ of a ``nice'' cone $\K$ (recall Eq.~\eqref{eq:Ax-cone-with-dual}), which means that a typical numerical method solving \eqref{eq:POP-relax} will result only in an \emph{approximate} certificate, rigorously proving that some polynomial \emph{close to, but not identical to} $p-\alpha$ is nonnegative over $\Delta$ \cite{MagronElDin2018}. On the dual side, a near-optimal feasible solution certifiably in $(\C^*)^\circ$ can be found, but this only yields an \emph{upper bound} on the optimal value of \eqref{eq:POP-relax}, that is, an upper bound on a lower bound on the optimal value of \eqref{eq:POP}. In summary, both \eqref{eq:POP-relax} and its dual may be solvable approximately using numerical methods, but neither numerical solutions provide a theoretically satisfactory, rigorously certified lower bound (with a nonnegativity certificate that is verifiable in rational arithmetic) on the optimal value of \eqref{eq:POP}. Instead, additional post-processing, in the form of some rounding procedure, is required to turn the numerical bounds and certificates into rigorous ones, which (in addition to being additional work) puts additional constraints on the quality of the numerical solution. See, e.g., \cite{PeyrlParrilo2008,KaltofenLiYangZhi2012,MagronElDin2018} for examples in the context of SOS certification.

The dual certificates described in this paper, and the reconstruction formula \eqref{eq:reconstruction-formula} resolve this problem in a simple, unifying framework, by allowing us to compute explicit, exact SOS, SONC, etc. representations, that is, near-optimal (exactly) feasible solutions to \eqref{eq:POP-relax}, from suitable near-optimal solutions of the dual problem for various cones $\mathscr{C}$. As an additional theoretical benefit, dual certificates provide ``parsimonious'' characterizations of polynomials in $\C$, in contrast with the usual representations of these cones as images of much higher dimensional cones. (Although, if necessary, those representations can also be constructed in closed form using the reconstruction formula.)

In the following subsections, we discuss these characterizations for the most commonly used tractable subcones of nonnegative polynomials. For ease of presentation, we focus on nonnegativity over $\Delta = \R^n$; the extension to the general case is straightforward as long as an appropriate Positivstellensatz is available. For instance, when $\Delta$ is compact, the results of the recent work \cite{KuryatnikovaVeraZuluaga2024} cover at least one appropriate Positivstellensatz using every technique considered in this section: SOS, SONC, DSOS, SDSOS and SAGE polynomials, and nonnegative AG functions. Throughout, we shall represent polynomials in the monomial basis, using the common shorthand $\vx^\vaf \defeq \prod_{i=1}^n x_i^{\alpha_i}$ to represent monomials.

\subsection{Sums-of-squares certificates}
\label{sec:SOS}
We return to \cref{ex:SOS} to develop some additional detail regarding H-certificates in sums-of-squares cones.

\begin{definition}
An $n$-variate polynomial of degree $2d$ is \emph{sums-of-squares} (SOS) if it can be written as a sum of squares of degree-$d$ polynomials. The set of such polynomials is denoted by $\Sigma_{n,d}$.
\end{definition}
The set $\Sigma_{n,d}\subseteq\R^N$ with $N=\binom{n+2d}{n}$ is a proper convex cone with a well-known semidefinite representation, briefly recalled for completeness in \cref{thm:SOS} below. We need the following notation: $\bS^M$ denotes the set of real, symmetric matrices of order $M$, and $\bS^M_+$ denotes the cone of positive semidefinite matrices within $\bS^M$. If $A(\cdot)$ is a linear map between finite dimensional real linear spaces, $A^*(\cdot)$ denotes its adjoint.
\begin{proposition}
\label{thm:SOS}
Given positive integers $d$ and $n$, denote by $m_d(x)$ the vector of monomials of degree up to $d$ in the variables $x=(x_1,\dots,x_n)$, let $M=\binom{n+d}{n}$ and $N=\binom{n+2d}{n}$. A polynomial $p:x\mapsto\sum_{i=1}^N p_{\vaf_i}x^{\vaf_i}$ is SOS if and only if
\begin{equation}
\label{eq:mXm}
p(x) = m_d(x)^\T\vX m_d(x)
\end{equation}
identically for some $\vX\in\bS^M_+$. In other words, there exists a linear map $A:\bS^M \to \R^N$ such that
\[ p \in \Sigma_{n,d} \Longleftrightarrow \exists\,\vX\in\bS^M_+: p = A(\vX). \]
The dual-SOS cone, often called the \emph{pseudo-moment cone}, in turn, also has a semidefinite representation:
\[ y \in \Sigma_{n,d}^* \Longleftrightarrow  A^*(y)\in\bS^M_+. \]
\end{proposition}

Thus, $\Sigma_{n,d}$ is clearly of the form $\K_A$ in \eqref{eq:Ax-cone-with-dual}, with $\bS^M_+$ playing the role of $\K$. Note that in general, $N \gg M$, that is, this representation characterizes the SOS cone as the linear image of a much higher-dimensional semidefinite cone. The linear map $A$ in \cref{thm:SOS} can be easily constructed; in matrix form it is a zero-one matrix that encodes which monomials can be obtained as a product of pairs of monomials.

Analogous statements can be formulated for polynomials represented in other bases, such as the Chebyshev basis of the first or second kind, which helps numerical conditioning, and also for weighted SOS polynomials nonnegative over semialgebraic sets $\Delta$; this requires no changes other than replacing $A^*$ with another suitable, easy-to-construct, linear map $\Lambda$ that depends on the bases and the polynomials representing $\Delta$ in Eq.~\eqref{eq:Delta}; see, e.g., \cite{Nesterov2000}.

It is well-known that $\bS^M_+$ is the hyperbolicity cone associated with the determinant function (a hyperbolic polynomial), and as a result, the function $f:X\mapsto-\log\det(X)$ is an $M$-LHSCB for this cone \cite{Guler1997}. Consequently, $\Sigma_{n,d}^*$ is also a hyperbolicity cone with the corresponding $M$-LHSCB $f_{\Sigma_{n,d}^*}:y\mapsto-\log\det\Lambda(y)$, and H-certificates defined using the Hessian of this function are rigorous nonnegativity certificates. Often, the operator $\Lambda$ corresponding to the bases and weights used to represent the cone maps rational vectors to rational matrices. This means that numerically computed H-certificates are automatically rational nonnegativity certificates that can be verified in polynomial time (in the bit model) in rational arithmetic \cite{DavisPapp2024}.

\subsection{SONC certificates}

\begin{definition}
We say that an $n$-variate polynomial $p$ is a \emph{circuit polynomial} if its support can be written as $\supp(p) = \{\vaf_1,\dots,\vaf_r,\vbeta\} \subset \N^n$, where the set $\{\vaf_1,\dots,\vaf_r\}$ is affinely independent and $\vbeta = \sum_{i=1}^r \lambda_i \vaf_i$ with some $\lambda_i>0$ satisfying $\sum_{i=1}^r \lambda_i = 1$. In other words, $\vbeta$ lies in the convex hull of the $\vaf_i$, and the scalars $\lambda_i$ are the corresponding (uniquely defined, strictly positive) barycentric coordinates of $\vbeta$. The support set of a circuit polynomial is called a \emph{circuit}.
\end{definition}

The cone of nonnegative circuit polynomials is closely related to the notion of power cones. The literature has conflicting notation and terminology; we shall use the following definitions, compatible with our earlier \cref{ex:ExpConeNoH}:
\begin{definition}
The \emph{(generalized) power cone} with \emph{signature} $\vlam = (\lambda_1,\dots,\lambda_r)$ is the convex cone defined as
\begin{equation*}
    \cP_\vlam \defeq \left\{ (\vv,z)\in \R_+^r \times \R \,\middle|\, |z|\leq \vv^\vlam \right\}.
\end{equation*}
Its dual cone is the \emph{(generalized) dual power cone}:
\[ \cP^*_\vlam \defeq \left\{ (\vv,z)\in \R_+^r \times \R \,\middle|\, |z|\leq \prod_{i=1}^r\left(\frac{v_i}{\lambda_i}\right)^{\lambda_i} \right\}. \]
\end{definition}

It can be shown that $\cP_\vlam$ and $\cP_\vlam^*$ are proper convex cones for every signature $\vlam\in (0,1)^r$ \cite{Chares2009,DresslerNaumannTheobald2021}.
With this notation, it is easy to characterize the subset of circuit polynomials nonnegative on $\R^n$:

\begin{proposition}[\cite{IlimanDeWolff2016,Papp2023}]\label{thm:SONC-main}
Let $p$ be an $n$-variate circuit polynomial satisfying $p(\vx) = \sum_{i=1}^r \fai \vx^{\vaf_i} + p_{\vbeta} \vx^{\vbeta}$ for some real coefficients $\fai$ and $p_\vbeta$, and suppose that $\vbeta = \sum_{i=1}^r \lambda_i \vaf_i$ with some $\lambda_i>0$ satisfying $\sum_{i=1}^r \lambda_i = 1$. Then $p$ is nonnegative if and only if $\vaf_i \in (2\N)^n$ and $\fai > 0$ for each $i$, and at least one of the following two alternatives holds:
\begin{enumerate}
\item $\vbeta \in (2\N)^n$ and $p_{\vbeta} \geq 0$, or
\item $\big((p_{\vaf_1},\dots,p_{\vaf_r}),p_{\vbeta}\big) \in \cP^*_\vlam$.
\end{enumerate}
\end{proposition}
\Cref{thm:SONC-main} has a simple interpretation: a circuit polynomial is nonnegative if and only if it is either a sum of nonnegative monomials, also known as \emph{monomial squares}, or if its nonnegativity can be directly proven by applying the weighted AM/GM inequality to its monomials, with weights $\vlam$.

\begin{definition}
\label{def:SONC}
    We say that a polynomial is a \emph{sum of nonnegative circuit polynomials}, or \emph{SONC} for short, if it can be written as a sum of monomial squares and nonnegative circuit polynomials. An explicit representation of a polynomial in this form is called a \emph{SONC decomposition}.
\end{definition}

We are ready to state the key representation theorem of SONC polynomials.
\begin{proposition}[\protect{\cite[Theorem 5]{Papp2023}}]
\label{thm:support}
Every SONC polynomial $p$ has a SONC decomposition in which every nonnegative circuit polynomial and monomial square is supported on a subset of $\supp(p)$.
\end{proposition}

Thus, in the notation of Section \ref{sec:projected-cones}, the cone of SONC polynomials of degree $\deg(p)$ supported on $\supp(p)$ is a cone of the form $\K_A \subseteq \R^{|\supp(p)|}$,
where the high-dimensional cone $\K$ is the Cartesian product of dual power cones (one for each circuit supported on $\supp(p)$) and nonnegative half-lines (one for each monomial square). The matrix $A$ of the linear transformation is a matrix of zeros and ones that tallies the coefficients of each monomial from the different circuits and monomial squares they appear in.

Since the dual cone $\Ks$ is the Cartesian product of a nonnegative orthant and power cones of the form $\cP_\vlam$ with different signatures $\vlam$, in order to apply the techniques of Section \ref{sec:projected-cones}, we only need an LHSCB for $\cP_\vlam$. Such barriers have already been found by a number of authors \cite{Chares2009,RoyXiao2022}. For example,

\[f_{\cP_\vlam}(\vv,z) = -\ln(v^{2\lambda} - z^2)-\sum_{i=1}^n (1-\lambda_i)\ln(v_i)
\]
is an $(n+1)$-LHSCB for this cone \cite{RoyXiao2022}. Note, however, that $\cP_\vlam$ is not a hyperbolicity cone for every $\vlam$, and we have already shown in \cref{ex:ExpConeNoH} that H-certificates do not work for, e.g., $\vlam=(2/3,1/3)$. Therefore, at least for general $\vlam$, B-certificates are needed for rigorous nonnegativity certification.

\subsection{Signomial functions and S-cones}

The concept of SONC polynomials can be straightforwardly generalized to nonnegative $n$-variate functions of the form
\begin{equation}
\label{eq:AGF}
p(\vx) = \sum_{\vaf\in\mathscr{{A}}}c_\vaf|\vx|^\vaf +  \sum_{\vbeta\in\mathscr{{B}}}d_\vbeta\vx^\vbeta \end{equation}
with exponents $\mathscr{A}\subseteq\R^n$ and  $\mathscr{B}\subseteq\N^n\setminus(2\N)^n$, using $|x|$ to denote component-wise absolute value. The exponents in $\mathscr{A}$ need not be integers, so this set encapsulates not only polynomials of $\vx$ and $|\vx|$, but also signomial functions \cite{ChandrasekaranShah2016}.
Analogously to circuit polynomials in the SONC setting, we define ``atomic'' functions in this set, called AG functions, for which nonnegativity holds if and only if it can be proven via a single application of the weighted AM/GM inequality:
\begin{definition}
Let $p$ be a function of the form \eqref{eq:AGF}. We say that
\begin{enumerate}
    \item $p$ is an \emph{even AG function} with \emph{support} $(\mathscr{A},\mathscr{B})$ if at most one $c_\vaf$ is negative and all $d_\vbeta$ are zero; and
    \item $p$ is an \emph{odd AG function} with \emph{support} $(\mathscr{A},\mathscr{B})$ if all $c_\vaf$ are nonnegative and at most one $d_\vbeta$ is nonzero.
\end{enumerate}
In addition, $p$ is called an \emph{AG function with support $(\mathscr{A},\mathscr{B})$} if $p$ is an even AG function or an odd AG function with the same support. Finally, the \emph{S-cone} $\SAB$ is the set of AG functions nonnegative on $\R^n$.
\end{definition}
 Katth{\"a}n, Naumann and Theobald have shown that, analogously to the SONC results, functions in $\SAB$ can be written as sums of AG functions with support $(\mathscr{A},\mathscr{B})$ \cite{KatthanNaumannTheobald2021}. They have also characterized both even and odd AG functions in terms of at most two relative entropy cone inequalities \cite[Theorem 2.8]{KatthanNaumannTheobald2021}. This representation is compatible with the concept of dual certificates. For example, an $n$-variate odd AG function given by \eqref{eq:AGF}
with support $(\mathscr{A},\{\vbeta\})$ is nonnegative on $\R^n$ if and only if there exists a $\nu\in\R_+^\mathscr{A}$ such that
\[ \sum_{\vaf\in\mathscr{A}}\nu_\vaf\vaf = \Big(\sum_{\vaf\in\mathscr{A}}\nu_\vaf\Big)\vbeta, \quad (d,\nu,e\vc)\in\cR_{|\mathscr{A}|}, \quad \text{and} \quad (-d,\nu,e\vc)\in\cR_{|\mathscr{A}|},\] and where $e=\exp(1)$ and $\cR_N$ represents the \emph{(vector) relative entropy cone}
\begin{equation}
\label{eq:relative-entropy-def}
\cR_N \defeq \operatorname{cl}\left\{ (u,\vv,\vw) \in \R \times \R^N_{++} \times\R_{++}^N \,\middle|\, u \geq \sum_{i=1}^N v_i\ln\left(\frac{v_i}{w_i}\right)\right\}.
\end{equation}
A similar characterization can be given for nonnegative even AG functions. As a result, membership in $\SAB$ can be B-certified using any known LHSCB for the dual of the relative entropy cone, $\cR_N^*$. This cone is also already well-studied, and it also has a suitable barrier function:
\begin{proposition}[\cite{ChenGoulart2023-ARXIVONLY}]
\label{thm:relative-entropy-dual}
The dual of the cone $\cR_N$ given in Eq.~\eqref{eq:relative-entropy-def} is the cone
\begin{equation*}
\cR_N^* \defeq \operatorname{cl}\left\{ (u,\vv,\vw) \in \R_{++} \times \R^N_{++} \times\R^N \,\middle|\, w_i \geq u\left(\ln\left(\frac{u}{v_i}\right)-1 \right)\;(i=1,\dots,N)\right\}.
\end{equation*}
The function
\[f_{\cR_N^*}(u,\vv,\vw) \defeq -\sum_{i=1}^N\ln\left(w_i-u\ln\left(\frac{u}{v_i}\right)+u\right)-N\ln(u)-\sum_{i=1}^N\ln(v_i)\]
is a $(3N)$-LHSCB for $\cR_N^*$.
\end{proposition}
Once again, B-certificates are required, because (similarly to the closely related exponential cone and its dual), relative entropy cones and their duals do not satisfy the relation \eqref{eq:HxinK}.

\subsection{Nonnegative bases, LP- and SOCP-representable subcones of SOS polynomials}
In an attempt to get around the high complexity of solving large-scaled SDPs associated with SOS certificates (\cref{thm:SOS}), various authors have suggested using polyhedral or second-order cone representable subcones of SOS polynomials \cite{Handelman1988,KuangGhaddarNaoumSawayaZuluaga2019,KuryatnikovaVeraZuluaga2024}. These include the following, in increasing order of complexity:
\begin{enumerate}
    \item Polynomials with nonnegative coefficients in a given basis that consists of trivially nonnegative polynomials over $\Delta$, e.g., when each basis polynomial is a product of some of the polynomials $q_i$ defining $\Delta$ in \eqref{eq:Delta};
    \item Polynomials of the form \eqref{eq:mXm} with a diagonally dominant $X$ whose diagonal is componentwise nonnegative (sometimes referred to as DSOS polynomials);
    \item Polynomials of the form \eqref{eq:mXm}, where $X$ is a matrix with factor-width no greater then $k$; i.e., $X$ is a sum of PSD matrices that are non-zero only in a single principal submatrix of order $k$ \cite{BlekhermanDeyShuSun2022}. (In the $k=2$ case, these are sometimes called SDSOS polynomials.)
\end{enumerate}
It is clear that each of these cones are of the form $\K_A$ where $\K$ is either the nonnegative orthant or a Cartesian product of PSD matrices of order at most $k$. In Cases (1) and (3) this is immediate from the definition. (Also recall that the cone of $2\times 2$ PSD matrices is linearly isomorphic to a three-dimensional second-order cone, making the SDSOS cone SOCP-representable.) In Case (2) we refer to the fact that diagonally dominant positive semidefinite matrices of order $n$ have $n^2$ explicitly known extreme rays, and every matrix in the cone can be written as a conic combination of these extreme vectors; see, e.g., \cite[Theorem 2]{BarkerCarlson1975} or \cite[Proposition 1]{Dahl2000}. Therefore, H-certificates obtained from the standard logarithmic barrier of $\K_A^*$ can be used to rigorously certify membership in this cone.

As with SOS polynomials, $\dim(\K)\gg\dim(\K_A)$ (particularly in Cases (1) and (3) above), and the matrix $A$ is an easily computable matrix with rational entries in the case of DSOS and SDSOS polynomials. It is usually a rational matrix in Case (1), too, unless the chosen basis polynomials have irrational coefficients. Thus, numerically computed H-certificates (see Section \ref{sec:computing}) are automatically rational vectors whose correctness can be verified in rational arithmetic if the certified polynomial $p-\alpha$ has rational coefficients.

\section{Discussion}
We presented a theoretical framework of cone membership certificates that allows us to interpret vectors in the dual cone as certificates of membership in the primal cone. The new certification scheme builds on the classic self-concordance theory of Nesterov and Nemirovski \cite{NesterovNemirovskii1994}, but it goes beyond the well-known result that the gradient map associated with an SCB defines a bijection between the interior of a cone and the interior of its dual. The key difference is that $\HC(\vb)$ and $\BC(\vb)$ are full-dimensional cones, which means that we can use numerical methods to identify any of the certificates, and still establish rigorous certification of membership. Similarly, in the case of a low-dimensional projection $\K_A$ of the high-dimensional cone $\K$, an exact preimage $\vx\in\K$ of a vector $\vb\in\K_A$ can be computed from any (numerically computed) dual certificate. The same infrastructure also allowed us to compute exact primal feasible solutions of general convex conic optimization problems from suitable dual feasible solutions, with duality gaps that approach zero as the dual feasible solutions approach optimality.

We also established that for many cones, some already existing optimization methods can be used to compute dual certificates. Specifically, feasible dual and primal-dual IPAs operating in appropriate neighborhoods of the central path, such as $\mathcal{N}(\eta,\tau)$ defined in \eqref{eq:nbd}, automatically work. For certain cones (e.g., linear images of symmetric cones), the existing dual methods automatically compute dual certificates, and these can now be extended to compute exact primal feasible solutions that converge to a primal optimal solution, complementing the dual feasible iterates of the method. For the most general case of nonsymmetric conic optimization, the IPA recently proposed by the authors \cite{PappVarga2025} can be applied.

In the future, we expect that this theory will lead to more efficient optimization methods for certain structured optimization problems. As outlined in the paper, optimization over a cone $\K_A$ can now circumvent the conventional ``lifting'' to optimization over $\K$ (also known as the ``extended formulation'') for a potential reduction in the solution time and especially in the required working memory. As long as $\Ks$ admits an LHSCB, we can directly optimize over $\KAs$ with the inherited barrier and reconstruct the desired primal solution in $\K_A$ after the fact.

Looking beyond optimization, the exactness of the reconstruction formula may be used to establish rigorous computer-generated proofs based on numerical optimization methods via rigorous nonnegativity certificates of polynomials. (E.g., exact SOS certificates of global lower bounds of a POP can be constructed from suitable feasible solutions of the Lasserre moment relaxation.) The last section of our paper lays the necessary theoretical groundwork for this; a more detailed study is the subject of a forthcoming paper.

%\appendix
\begin{comment}
\section*{Acknowledgments}
We thank the anonymous referees for their careful reading of the manuscript and their helpful comments on the presentation of the material.
\end{comment}

\bibliographystyle{siamplain}
\bibliography{dualcert}

\end{document}